\documentclass[10pt]{article}
\usepackage[utf8]{inputenc}
\usepackage[english]{babel}
\usepackage{amsmath,amsthm,amsfonts}
\usepackage{graphicx} 
\usepackage{bm}
\usepackage{tikz}
\usepackage{multirow}
\usepackage{color}
\usepackage{placeins}
\usepackage{indentfirst}

\newtheorem{theorem}{Theorem}
\newtheorem{lemma}[theorem]{Lemma}
\graphicspath{ {./img/} }

\renewcommand{\div}{\mathop{\rm div}\nolimits}

\usepackage{todonotes} 

\usepackage{enumitem}

\begin{document}

\title{Mixed Generalized Multiscale Finite Element Method for Flow Problem in Thin Domains}

\author{
Denis Spiridonov 
\thanks{Multiscale model reduction Laboratory, North - Eastern Federal University, Yakutsk, Republic of Sakha(Yakutia), Russia, 677980. {\tt d.stalnov@mail.ru}.}
\and
Maria Vasilyeva 
\thanks{Department of Mathematics and Statistics, Texas A\&M University, Corpus Christi, Texas, USA. Email: {\tt maria.vasilyeva@tamucc.edu}.}
\and
Min Wang
\thanks{Department of Mathematics, Duke University, Durham, NC 27705, USA {\tt wangmin@math.duke.edu}.}
\and
Eric T. Chung \thanks{Department of Mathematics,
The Chinese University of Hong Kong (CUHK), Hong Kong SAR. Email: {\tt tschung@math.cuhk.edu.hk}.}
}

\maketitle

\begin{abstract}
In this paper, we construct a class of Mixed Generalized Multiscale Finite Element Methods for the approximation on a coarse grid for an elliptic problem in thin two-dimensional domains. We consider the elliptic equation with homogeneous boundary conditions on the domain walls.  For reference solution of the problem, we use a Mixed Finite Element Method on a fine grid that resolves complex geometry on the grid level.  To construct a lower dimensional model, we use the Mixed Generalized Multiscale Finite Element Method, which is based on some multiscale basis functions for velocity fields. The construction of the basis functions is based on the local snapshot space that takes all possible flows on the interface between coarse cells into account. In order to reduce the size of the snapshot space and obtain the multiscale approximation, we solve a local spectral problem to identify dominant modes in the snapshot space. We present a convergence analysis of the presented multiscale method. Numerical results are presented for two-dimensional problems in three testing geometries along with the errors associated to different numbers of the multiscale basis functions used for the velocity field. Numerical investigations are conducted for problems with homogeneous and heterogeneous properties respectively. 
\end{abstract}

\section{Introduction}

In this paper, we consider a class of multiscale problems in thin domains with complex geometries.  Such problems occur in many real world applications, for example, blood flow in vessels, flow in rough fractures, etc.  Problems in thin domains  have been studied in \cite{quarteroni2004mathematical, nachit2013asymptotic, oshima2001finite}. Such problems also occur in the field of engineering modeling, examples include the modeling of fluid flow in pipe networks, oil and gas production, geothermal sources, etc. Moreover, in tasks such as filtration modeling in reservoirs, the computational domain under consideration may be of compounded complexity, as they usually contain many cracks of various shapes and sizes. The length of a crack is usually much greater than its thickness, therefore, in some cases, a fracture can be considered separately \cite{formaggia2014reduced, martin2005modeling, d2008coupling}. 

Thin domain problems are very complicated and often lead to computational system with large number of unknowns. This is due to the fact that very fine grid is required in numerical schemes such as the Finite Element Method\cite{zienkiewicz1977finite, zienkiewicz2005finite, dhatt2012finite} and the Mixed Finite Element Method\cite{raviart1977mixed, boffi2013mixed, gatica2014simple} to accurately describe the heterogeneities and complex shape of the media. One way to improve the computational efficiency is to conduct parallel computing\cite{hussain2013parallel, cliffe2000parallel, douglas1993parallel}. Another option is to use the multiscale model reduction techniques to reduce the size of the computational system\cite{hernandez2014high, chung2018multiscale, allaire2005multiscale}, which we will adopt in our paper. 

The standard multiscale model reduction techniques include homogenization and numerical homogenization methods. The homogenization method can be used to construct the approximation of the solution on a coarse grid by calculating the effective properties of the material.
The methods of numerical homogenization give macroscopic laws and parameters on the basis of local computations, but such approaches are usually based on a priori formulated assumptions \cite{bakhvalov2012homogenisation, talonov2016numerical, hales2015asymptotic, engquist2008asymptotic, tyrylgin2019numerical}. Another type of such technique is the multiscale method. This method constructs multiscale basis functions to extract information of the domain on micro-level and then bridges the information between micro and macro scales \cite{efendiev2009multiscale, efendiev2013generalized, chung2015mixed}.

Multiscale methods are widely used to solve various problems in domains with complex heterogeneities. Some of the popular methods are the multiscale finite element method(MsFEM)\cite{hou1997multiscale, efendiev2009multiscale, chung2010reduced}, mixed multiscale finite element method(Mixed MsFEM)\cite{chen2003mixed, aarnes2004use}, generalize multiscale finite element method(GMsFEM)\cite{efendiev2013generalized, spiridonov2019generalized, chung2014generalized,chung2020generalized,wang2020generalized}, heterogeneous multiscale methods\cite{engquist2007heterogeneous, abdulle2012heterogeneous, weinan2003heterogeneous}, multiscale finite volume method (MsFVM)\cite{hajibeygi2008iterative, tyrylgin2019embedded, lunati2006multiscale}, constraint energy minimizing generalized multiscale finite element method (CEM-GMsFEM)\cite{chung2018constraint, chung2018constraint} and etc. 
Recently, in \cite{ chung2018non, vasilyeva2019nonlocal, vasilyeva2019constrained}, the authors present a special design of the multiscale basis functions to solve problems in fractured porous media, which obtains the basis functions based on the constrained energy minimization problems and nonlocal multicontinuum (NLMC) method. The multiscale modeling of Darcy flow in perforated domain are presented in \cite{chung2016mixed} while a similar problem in fractured domain is considered in \cite{spiridonov2018multiscale}. There are also many applications of the Mixed MsFEM to the Darcy flow\cite{gulbransen2009multiscale, duran2019multiscale}. 
We can also highlight the Mixed MsFEM implementation for the Darcy-Forchheimer model\cite{spiridonov2019mixed} for flows in fractured media, where Darcy flow has a non-linear nature since the velocity in the fractures is much higher than in the matrix.   

In this paper we present a model for Darcy flow in heterogeneous thin domain where the length of the medium is much larger than the width. To provide numerical simulations, we use a multiscale model reduction technique. The reduction is achieved through the application of the Mixed Generalized Multiscale Finite Element Method(Mixed GMsFEM)\cite{chung2015mixed, chung2019mixed}. The algorithm of the Mixed GMsFEM has two stages: the offline stage and the online stage. At the offline stage, we obtain the multiscale basis functions by solving a spectral problem on local domains. The spectral problems are solved in the snapshot space that takes into account all possible flows on the interface between coarse cells.  By using the multiscale basis functions, we obtain an offline space to approximate the flow part of the solution. For the pressure part, we use the piece-wise constant functions to approximate the solution on the coarse grid. At the online stage we solve our problem on the coarse grid using the offline space. 
In this paper, we will provide our experiment results for different shapes of computational domain. We present the dependence of the approximation accuracy of the Mixed GMsFEM solution on the shape of the computational domain.  To show this dependence, we will compare multiscale solutions with the reference solutions. The reference solution is taken to be the solution on the fine grid obtained with Mixed Finite Element Method.

The paper is organized as follows. In Section 2, we present the problem formulation with the approximation on the fine grid given by Mixed Finite Element Method that resolve complex geometry at the grid level.   In Section 3, the construction of the Mixed GMsFEM solutions to the problems in thin domains is presented. In particular, the multiscale basis functions for velocity field on the coarse interfaces are constructed.  Convergence analysis of the presented method is presented in Section 4. In Section 5, we present numerical results for two-dimensional test problems. We consider three testing geometries and investigate the numerical errors for both pressure and velocity using different numbers of the multiscale basis functions.

\section{Problem formulation}

Let $\Omega$ be a two-dimensional thin domain $\Omega \subseteq	 \mathbb{R}^2$, where thickness is much smaller than length of the domain (see Figure \ref{domain}).  
Boundaries $\Gamma_{\text{in}}$ and $\Gamma_{\text{out}}$ are the inlet and outlet boundaries of the thin domain $\Omega$, while $\Gamma_{\text{w}} := \partial \Omega / (\Gamma_{\text{in}} \cup \Gamma_{\text{out}})$.  

\begin{figure}[h!]
\centering
\includegraphics[width=0.9\linewidth]{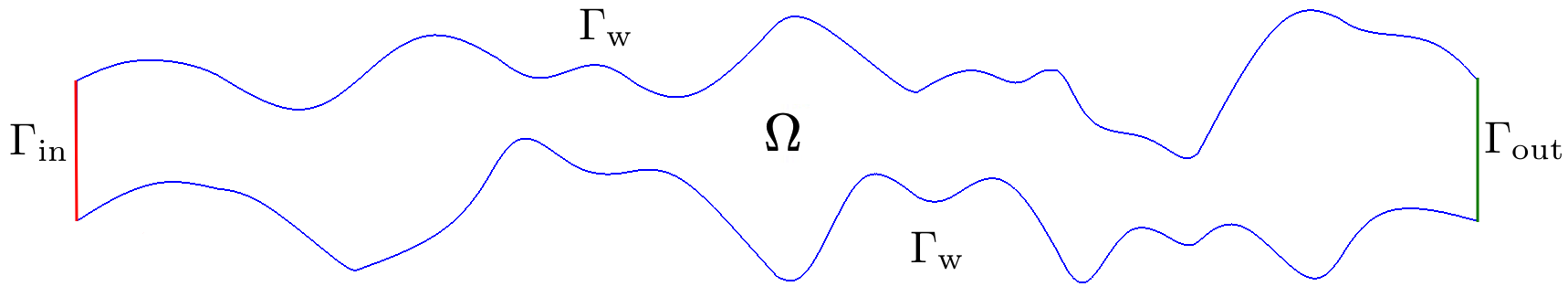} 
\caption{
Illustration of the thin domain $\Omega$, 
$\partial \Omega = \Gamma_{\text{in}} \cup \Gamma_{\text{out}}  \cup \Gamma_{\text{w}}$. 
}
\label{domain}
\end{figure}
 
We consider an elliptic equation on the  thin domain $\Omega$ 
\begin{equation}
\label{eq:mm}
\begin{split}
\kappa^{-1} u + \nabla p &= 0, \quad x \in \Omega, \\
\nabla \cdot u & = f,  \quad x \in \Omega,
\end{split}
\end{equation}
where $u$ is the velocity field,  $p$ is the pressure, $\kappa$ is the heterogeneous coefficient and $f$ is the source term. On $\Gamma_{\text{w}}$, a homogeneous boundary conditions is imposed to the system \eqref{eq:mm}:
\begin{equation}
\label{eq:bcw}
u \cdot n = 0, \quad x \in \Gamma_{\text{w}},
\end{equation}
and the non-homogeneous boundary conditions are imposed to the pressure on the inlet $\Gamma_{\text{in}}$ and outlet $\Gamma_{\text{out}}$ boundaries
\begin{equation}
\begin{cases} 
      p =p_1 &  x \in \Gamma_{\text{in}} \\
       p =p_2 &  x \in \Gamma_{\text{out}} \\
   \end{cases},
    \label{eq:bc}
\end{equation}
where $n$ is the outward normal vector to the domain boundary.

\noindent\noindent\textbf{Variational formulation}

We present a weak formulation of problem \eqref{eq:mm}, \eqref{eq:bcw}, \eqref{eq:bc} as follows. 
Let 
\[
V := \{v\in H(\text{div}; \Omega) \ : \ v \cdot n = 0 \text{  on  } \Gamma_{\text{w}} \}
\]
and $Q := \mathcal{L}^2(\Omega)$.
We have  the following variational formulation:  Find $(u, p) \in V \times Q$ such that
\begin{equation}
\begin{split}
-\int_{\Omega}  \kappa^{-1} u \, v \, dx
+ \int_{\Omega} p \,\nabla \cdot v \, dx
&= \int_{\Gamma_{\text{in}}} p_1 \ v \cdot n \ ds 
+ \int_{\Gamma_{\text{out}}} p_2 \ v \cdot n \ ds , \quad  \forall v \in V, \\
\int_{\Omega} q \, \nabla \cdot u \, dx 
&= \int_{\Omega} f \, q \, dx, \quad \forall q \in Q.
\end{split}
\end{equation}

\begin{figure}[h!]
\centering
\includegraphics[width=0.9\linewidth]{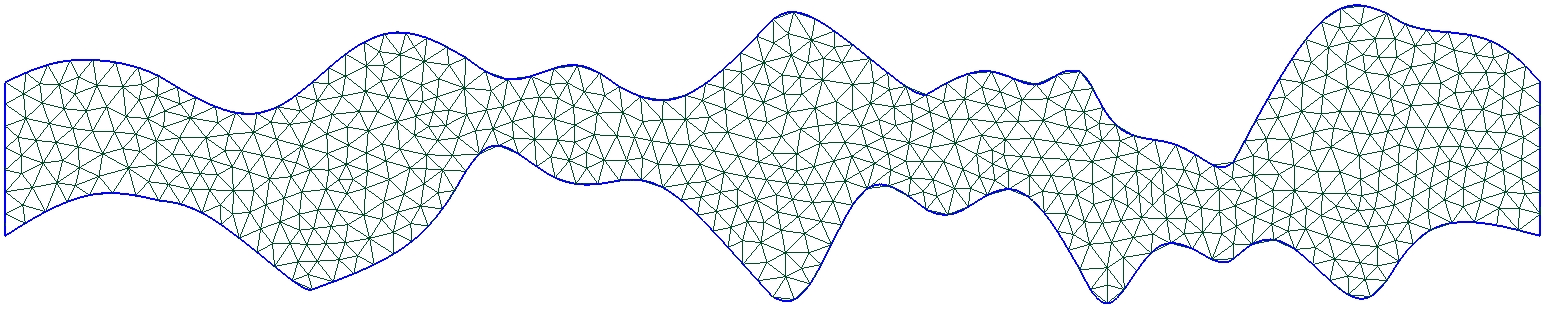} 
\caption{
Illustration of the fine grid $\mathcal{T}_h$ for domain $\Omega$. 
}
\label{finemesh}
\end{figure}

\noindent\textbf{Fine grid approximation}

We construct a fine grid $\mathcal{T}_h$ that resolves the complex geometry on a grid level (see Figure \ref{finemesh}).   
Let $V_h \subset V$ and $Q_{h}\subset Q$ be the finite element subspaces on the fine grid $\mathcal{T}_h$. In particular, 
we use the lowest order Raviart-Thomas elements for velocity and piece-wise constant elements for pressure respectively.  
Then,  we have the following discretized problem on the fine grid: Find $(u_h, p_h) \in V_h \times Q_h$ such that
\begin{equation}
\label{fine_problem}
\begin{split}
a(u_h, v_h) + b(p_h, v_h) & = g(v_h),  \quad \forall v_h \in V_h, \\
b(q_h,  u_h) & = f(q_h), \quad \forall q_h \in Q_h.
\end{split}
\end{equation}
where 
\[
\begin{split}
&a(u_h, v_h) := -\int_{\Omega}  \kappa^{-1} \ u_h \ v_h \ dx,  \qquad \qquad \qquad \quad
b(p_h, u_h) := \int_{\Omega} p_h \ \nabla \cdot u_h \ dx, 
\\
&g(v_h) := \int_{\Gamma_{\text{in}}} p_1 \ v_h \cdot n \ ds 
+ \int_{\Gamma_{\text{out}}} p_2 \ v_h \cdot n \ ds , \quad 
\ f(q_h) := \int_{\Omega} f \, q_h \, dx. 
\end{split}
\]
We can then write the discrete system \eqref{fine_problem} in the  following matrix form 
\begin{equation}
 \begin{pmatrix}
A_h & B_h^T \\
 B_h & 0
 \end{pmatrix} \begin{pmatrix}
 U_h \\
 P_h
 \end{pmatrix} = \begin{pmatrix}
G_h \\
F_h
 \end{pmatrix},
\end{equation} 
\[
A_h = \{a_{ij} = a(\psi_i, \psi_j) \},  \quad
B_h = \{b_{ij} = b(\phi_i, \psi_j) \}, 
\]\[
G_h = \{g_j  = g(\psi_j) \},   \quad 
F_h = \{f_j  = f(\phi_j) \}.
\]
where $\psi_m$ and $\phi_m$ are the basis functions for velocity and  pressure on the fine grid respectively.

\section{Mixed Generalized Multiscale Finite Element Method}

In this section,  we describe the coarse grid approximation for problem \eqref{eq:mm} 
using the Mixed Generalized Multiscale Finite Element Method (Mixed GMsFEM). We describe the construction of the approximation on the coarse grid with the use of the multiscale basis functions for the velocity field. 
Since we consider the problem in thin domains with no-flow conditions on the wall boundary $\Gamma_{\text{w}}$, 
we only
compute multiscale basis functions for vertical edges of the coarse grid.  

\begin{figure}[h!]
\centering
\includegraphics[width=0.9\linewidth]{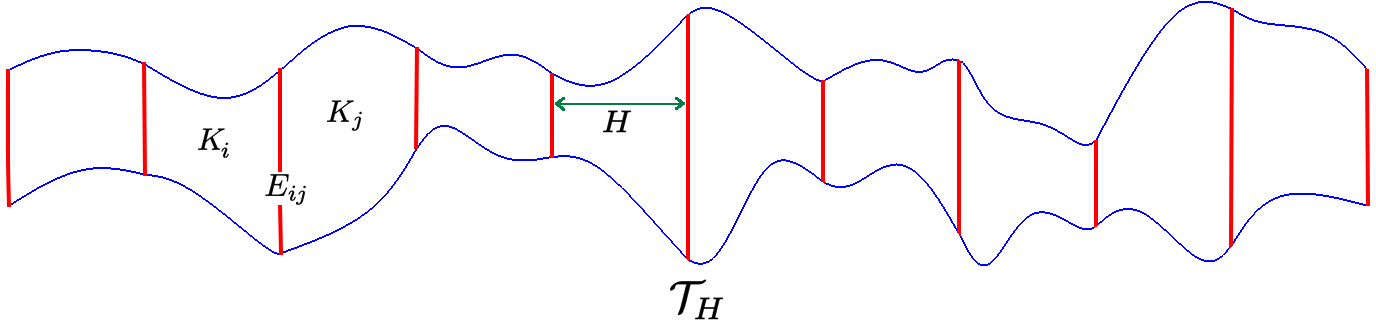} 
\caption{Illustration of the coarse grid, $\mathcal{T}_H$ 
}
\label{coarsemesh}
\end{figure}

Let $\mathcal{T}_H$ be a coarse grid of the computational domain $\Omega $ with a coarse grid size $H$
\[
\mathcal{T}_H = \cup_{i = 1}^{N_C} K_i,
\]
and $\mathcal{E}_H  = \cup E_{ij} $ is the set of all interfaces between two cells of the coarse grid,  $N_E$  and $N_C$ are the number of edges and cells of the coarse grid respectively.  
In particular, we construct the  multiscale basis functions for the velocity field following the Mixed GMsFEM in the local domains $\omega_{ij} := \{ K_i \cup K_j : K_i \cap K_j = E_{ij}, \ K_i, K_j \in \mathcal{T}_H\}$ that associated with the coarse edge $E_{ij} \in \mathcal{E}_H$ (see Figure. \ref{coarsemesh}).  Without lose of generality, we use a single index to denote local domains and coarse grid edges $\omega_i$ and $E_i$ from now on. 

We first need to construct the multiscale space for velocity field
\begin{equation}
V_H = 
\text{span} \{ \psi^{(i)}_l: \ 1 \leq i \leq N_E, \ 1 \leq l \leq M^{(i)} \}, 
\end{equation}
by expanding the multiscale basis function $\psi^{(i)}_l$ which are computed in the local domain $\omega _i$. $M^{(i)}$ denotes the number of basis functions in each local domain. We then take the space $Q_H$, which is constituted by piece-wise constant functions on the coarse cells, for pressure.
A coarse-scale solution can thus be obtained by solving the following problem: Find $(u_H, p_H) \in V_H \times Q_H$  such that
\begin{equation}
\label{coarse_problem}
\begin{split}
a(u_H,  v_H) + b(p_H, v_H) &= g(v_H),  \quad \forall v_H \in V_H, \\
b(q_H, u_H) &=  f(q_H), \quad \forall q_H \in Q_H.
\end{split}
\end{equation}

The construction of the multiscale basis functions $\psi^{(i)}_l \in \omega_i$ for the velocity field contains several steps: 
\begin{enumerate}[label=Step \arabic*:]
\item Generate the snapshot space in local domain $\omega_i$ for all possible flows through the edge $E_i$.
\item Obtain the multiscale basis functions $\psi_l^{(i)}$ by solving the local spectral problem in the local snapshot space associated to $\omega_i$. 
\end{enumerate}

We note that, the spectral problem is used to extract dominant modes in a snapshot space and reduce the dimension of the space used to approximate the velocity field.

\noindent\textbf{Snapshot space. }

To construct a snapshot space in $\omega_i$ we solve the following local problem: Find ($ \varphi^{(i)}_j,  \eta^{(i)}_j ) \in V^{(i)}_h \times Q^{(i)}_h $ such that
\begin{equation}
\label{eq:snapshot_problem}
\begin{split}
- \int_{\omega_i} \kappa^{-1} \varphi^{(i)}_j v \, dx 
+ \int_{\omega_i} \eta^{(i)}_j \, \nabla \cdot v \, dx 
& = 0, 
\quad \forall v \in V^{(i)}_h, \\
\int _{\omega _i} r \ \nabla \cdot \varphi^{(i)}_j, dx 
& = \int _{\omega _i} c^{(i)}_j \ r \ dx, 
\quad \forall r \in Q^{(i)}_h ,
\end{split}
\end{equation}
with boundary condition
\begin{equation} 
\label{eq:snapshot_problem_bc1}
\varphi^{(i)}_j \cdot n = 0, \quad x \in \partial \omega_i,
\end{equation}
and additional boundary condition on the coarse edge $E_i$
\begin{equation} 
\label{eq:snapshot_problem_bc2}
\varphi^{(i)}_j \cdot n = \delta_j, \quad x \in E_i, 
\end{equation}
where $j = 1, ...,J^{(i)}$,  $V_h^{(i)}: = V_h(\omega_i) $ and $Q_h^{(i)} : = Q_h({\omega_i})$ and  $n$ is the normal vector to the boundary.  
Here $c^{(i)}_j  := \frac{|e_j|}{|\omega_i|}$ is chosen by a compatibility condition,  where 
$|e_j | $ is the length of the fine grid edge $e_j$,   
$|\omega_i|$ is the volume of the local domain $\omega _i$,  
$J^{(i)}$  is the number of fine grid edges $e_j$ on $ E_i$,  $E_i = \cup^{J^{(i)}}_{j=1} e_j $, 
$\delta _j$ is a piece-wise constant function defined on $E_i$ which takes the value $ 1 $ on $e_j$ and $ 0 $ in all other edge of the fine grid. 

We further define the local snapshot space as
\[
V_{\text{snap}}^{(i)} = \text{span} \{ \varphi^{(i)}_j, \ 1 \leq j \leq J^{(i)} \}.
\]

\noindent\textbf{Multiscale space}

We  perform dimension reduction of the snapshot space $V_{\text{snap}}^{(i)} $ by solving the following spectral problem: solve for eigen pairs $\lambda$ and $\Psi^{(i)}$ for 
\begin{equation}\label{spectral}
a^{(i)}(\Psi^{(i)}, v) = \lambda \ s^{(i)}(\Psi^{(i)},  v),  \quad 
\forall v \in V_{\text{snap}}^{(i)},
\end{equation}
where 
\[
a^{(i)}(u, v) 
:= \int_{E_i} \kappa^{-1} (u \cdot n)(v \cdot n) \ ds, 
\quad 
s^{(i)}(u, v) 
:= \int _{\omega _i} \kappa^{-1} u \ v \ dx + \int _{\omega _i} \nabla \cdot u \ \nabla \cdot v\ dx.
\]

This is equivalent to solving the discrete system:
\begin{equation} 
\label{eq8}
\tilde{A}^{(i)} \Psi^{(i)} = \lambda^{(i)} \  \tilde{S}^{(i)} \Psi^{(i)}, 
\end{equation}
where 
\[
\tilde{A}^{(i)} = R^{(i)}_{\text{snap}} A^{(i)} (R^{(i)}_{\text{snap}})^T,  \quad 
\tilde{S}^{(i)} = R^{(i)}_{\text{snap}} S^{(i)} (R^{(i)}_{\text{snap}})^T,   \quad 
R^{(i)}_{\text{snap}} = [ \varphi^{(i)}_1, ...,  \varphi^{(i)}_{J^{(i)}}]^T
\]
with 
$A^{(i)} = \{a_{mn} = a^{(i)}(\psi_m, \psi_n)\}$ and 
$S^{(i)} = \{s_{mn} = s^{(i)}(\psi_m, \psi_n)\}$.  

\begin{figure}[h!]
\centering
\includegraphics[width=0.9\linewidth]{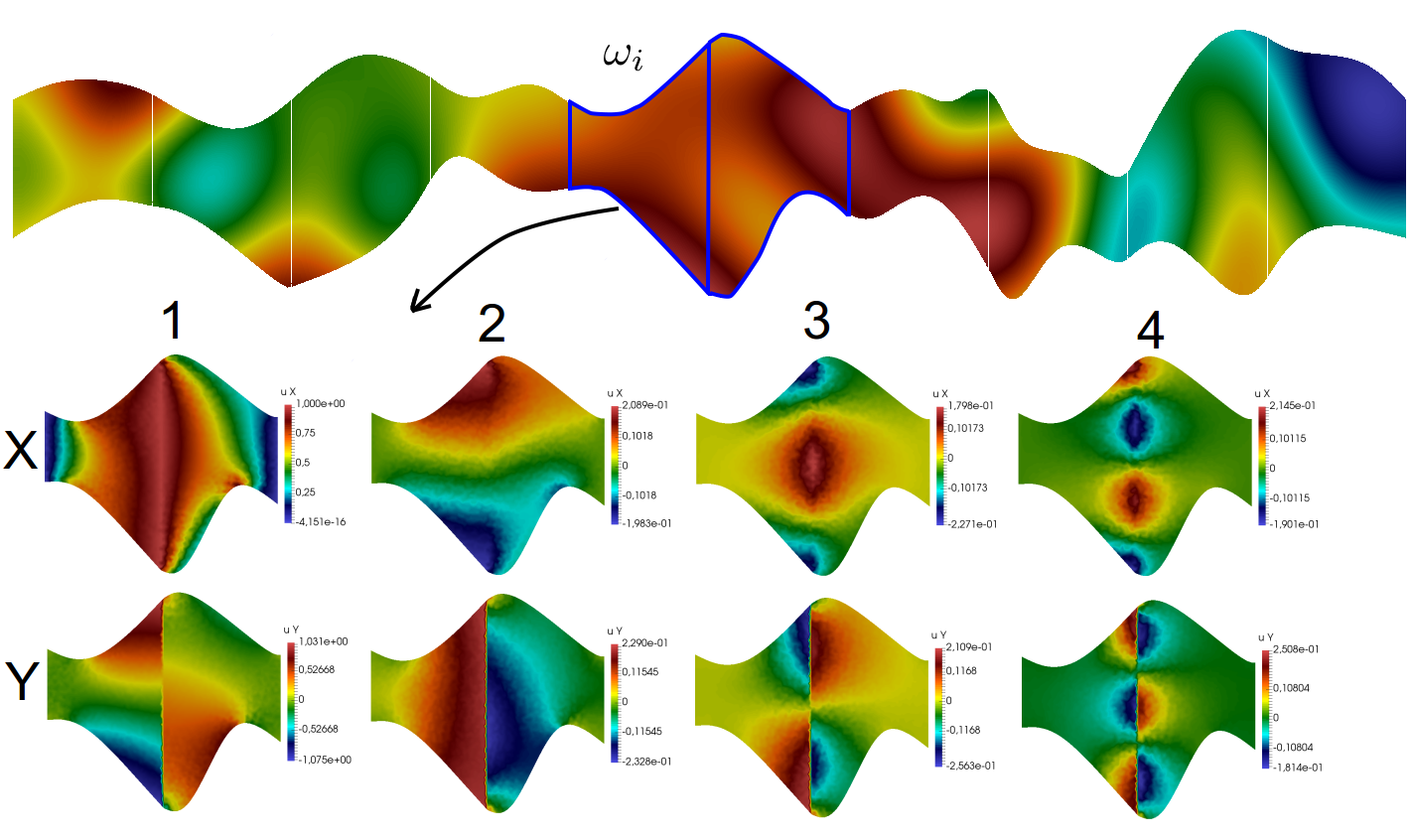} 
\caption{Illustration of the first four multiscale basis functions in local domain}
\label{msbasis}
\end{figure}

For the construction of the multiscale space, we sort the eigenvalues in ascending order. We select the first $M^{(i)}$ eigenvalues and take the corresponding eigenvectors $\psi^{(i)}_l = (R^{(i)}_{\text{snap}})^T \Psi^{(i)}_l $ as basis functions, $l = 1,...,M^{(i)}$.
The first four multiscale basis functions for velocity are presented in Figure \ref{msbasis}.

Finally,  we obtain the following multiscale space for velocity by expanding all local multiscale basis 
\begin{equation}
V_H = \text{span} \{ \psi^{(i)}_l,  \ 1 \leq l \leq M^{(i)},  \ 1 \leq i \leq N_E \}.
\end{equation}

\noindent\textbf{Coarse grid system}

For the construction of the coarse grid system, we define a projection matrix
\begin{equation} 
\label{eq13}
R = \begin{bmatrix}
R_u & 0 \\
0 & R_p 
\end{bmatrix}, 
\end{equation}
where 
\[
R_u = [ 
\psi^{(1)}_1, ..., \psi^{(1)}_{M^{(1)}}, ..., 
\psi^{(N_E)}_1, ..., \psi^{(N_E)}_{M^{(N_E)}} ], 
\]  
and  $R_p$ is the projection matrix for pressure, where we set one for each fine grid cell in the current coarse grid cell and zero otherwise.  

Using the projection matrix, we obtain the following coarse scale system in matrix form 
\begin{equation}
\label{eq:coarse}
 \begin{pmatrix}
 A_H & B^T_H \\
 B_H & 0
 \end{pmatrix} \begin{pmatrix}
 U_H \\
 P_H
 \end{pmatrix} = \begin{pmatrix}
G_H \\
F_H
 \end{pmatrix},
 \end{equation}
 where
 \begin{equation}
 A_H = R_u A_h R^T_u ,  \quad 
 B_H = R_u B_h R^T _p, \quad
 G_H = R_u G_h \quad 
 F_H = R_p F_h.
 \end{equation}
After solving the coarse grid system \eqref{eq:coarse},  we reconstruct the fine grid velocity field using the projection matrix, $U_{\text{ms}} = R^T_u U_H$.

\section{Convergence Analysis}

In this section, we prove the convergence of the Mixed GMsFEM for problems in thin domains with no flow boundary conditions on wall boundary. This analysis is based on our previous papers \cite{MixedGMsFEM, chung2016mixed}. 

We define the following norms and notations:
\begin{itemize}
\item  $||q||_{L^{2}(\Omega)}^{2}=\int_{\Omega}q^{2}\ dx$ is the $L^2$ norm for a scalar function $q \in L^2(\Omega)$;
\item $||v||_{\kappa^{-1}, \Omega}^{2} = \int_{\Omega} \kappa^{-1}|v|^{2}\ dx$ is the weighted $L^2$ norm for vector function $v$;
\item $||v||_{{H(\div,\Omega)};\kappa^{-1}}^{2} = ||v||_{\kappa^{-1}, \Omega}^{2} + ||\nabla \cdot v||_{L^{2}(\Omega)}^{2}$ is the norm in the Sobolev space $H(\div,\Omega)$ for vector function $v$ with $v \in [L^2(\Omega)]^2$ and $\div(v) \in L^2(\Omega)$;
\item $\alpha \preceq \beta$ means there is a constant $C>0$ such that $\alpha \leq C\beta$;
\end{itemize}

\begin{lemma}
Let $(u_h, p_h)\in V_h \times Q_h$ be fine solutions of (\ref{fine_problem}) and $\hat{u}$ be the weak fine-scale solution of the following problem
\begin{equation}
\label{snap-project}
\begin{split}
\kappa^{-1} \hat{u}+\nabla \hat{p} &=0, \quad x \in K,\\
\nabla\cdot \hat{u} &=\bar{f}, \quad x \in K,
\end{split}
\end{equation}
subject to
\[
\hat{u}\cdot n = 0, \quad x \in \Gamma_{\text{w}},\quad 
\hat{u}\cdot n =u_h\cdot n, \quad x \in  \partial K / \Gamma_{\text{w}},
\]
and
\[
\int_{K}\hat{p}  \ dx= \int_{K} p_h \ dx,   
\]
where $K$ is any coarse cell in $\mathcal{T}_H$ and $\bar{f}=\frac{1}{|K|}\int_{K} f \ dx$ is the average value of $f$ over $K$.   
Assume that 
$\kappa_{\text{min,K}}^{-1} : = \min{\{\kappa^{-1}(x)\ |\ x\in K \}}<\infty$ for $\forall K \in \mathcal{T}_H$,  then we have
\begin{equation}
\label{lemma1}
||u_h-\hat{u}||_{\kappa^{-1}, \Omega}^2
\preceq
\underset{K\in \mathcal{T}_H}{\max} \left( \kappa_{\text{min,K}}^{-1} \right)
\sum_{i} || f - \bar{f}||_{L^2(K_i)}^2. 
\end{equation}
\end{lemma}

\begin{proof}
Notice that, by construction,  $\hat{u} \in V_h \bigcap V_{\text{snap}}$. 
Then, by subtracting variational form of \eqref{snap-project} from \eqref{fine_problem}, we have
\begin{equation}
\label{proj_snap}
\begin{split}
\int_{K} \kappa^{-1}(u_h - \hat{u})\cdot w_h  \ dx
- \int_{K} (p_h-\hat{p}) \ \nabla \cdot w_h \ dx  
& = 0,  \quad \forall w_h \in V_h(K),\\
\int_{K} q_h \ \nabla \cdot(u_h - \hat{u}) \ dx 
&= \int_K ( f - \bar{f} ) \ q_h \ dx,  \quad \forall q_h \in Q_h(K).
\end{split}
\end{equation}
Taking $w_h = u_h - \hat{u}$ and $q_h = p_h - \hat{p}$ in (\ref{proj_snap}), we can then get
\begin{equation}
\label{eq:5.4}
\int_K \kappa^{-1} (u_h - \hat{u})^2   \ dx
= \int_K (f - \bar{f})(p_h - \hat{p}) \ dx.
\end{equation}

Since functions contained in Raviart-Thomas space satisfies the inf-sup condition, we have
\begin{equation}
\label{inf-sup}
||q_h||_{L^2(K)} \preceq \underset{w_h \in V_h(K)}{\text{sup}}\frac{\int_K q_h \ \nabla \cdot w_h \ dx}{||w_h||_{H(\div; K)}},
\quad \forall q_h\in Q_h(K).
\end{equation}
Combining condition \eqref{inf-sup} with \eqref{proj_snap},  we can get:
\[
||p_h - \hat{p}||_{L^2(K)} 
\preceq 
\kappa_{\min, K}^{-\frac{1}{2}} ||u_h-\hat{u}||_{\kappa^{-1}, K}.
\]
Then by \eqref{eq:5.4}
\begin{equation}
||u_h-\hat{u}||_{\kappa^{-1}, K} 
\preceq 
\kappa_{\min,  K}^{-\frac{1}{2}} ||f-\bar{f}||_{L^2(K)}.
\end{equation}
Summing up the results for all coarse blocks, we obtain the desired conclusion.
\end{proof}

In the next result, we will prove that our mixed GMsFEM satisfies an inf-sup condition,
which is important for the convergence analysis of the method.

\begin{theorem}
For all $p \in Q_H$,  we have 
\begin{equation}
\label{eq:inf_sup}
||p||_{L^2(\Omega)} 
\preceq  
C_{\text{infsup}} \underset{u \in V_H}{\sup}{\frac{\int_{\Omega} p \ \nabla \cdot u \ dx}{||u||_{H(\div; \Omega); \kappa^{-1}}}}
\end{equation}
where 
$C_{\text{infsup}} := 
\left(
\underset{i}{\max} \ \underset{l}{\min} 
\int_{\omega_i} \kappa^{-1} \psi^{(i)}_l \cdot \psi^{(i)}_l \ dx + 1
\right)^{\frac{1}{2}}$.
\end{theorem}

\begin{proof}
Consider the problem $\Delta \xi = p$ in $\Omega$ with zero conditions $\partial \xi / \partial n = 0$ on $\partial \Omega$.  
Let $v = \nabla \xi$,  then we have $\nabla \cdot v = p$. 
By the Green's identity and the regularity theory,  we have  \cite{MixedGMsFEM, chung2016mixed}
\begin{equation}
\label{eq:infsup1}
H \sum_{i=1}^{N_C} \int_{\partial K_i} (v\cdot n)^2 \ ds 
\leq ||p||^2_{L^2(\Omega)}. 
\end{equation}

Let 
\[
u = \sum_{i=1}^{N_E} b_i \tilde{\psi}^{(i)}_l, \quad 
b_i = \int_{E_i} u \cdot n \ ds, 
\]
where $\tilde{\psi}^{(i)}_r$ is the normalized basis functions, in this proof only,  with the property $\int_{E_i}  \tilde{\psi}^{(i)}_r \cdot n \ ds = 1$.
Therefore using property of normalized basis functions, we have 
\begin{equation}
\label{eq:infsup2}
\int_{\Omega} p^2 dx 
= \int_{\Omega} p \ \nabla \cdot v \ dx 
= \sum_{i=1}^{N_E} \int_{E_i} [p]  (u \cdot n) \ ds 
= \sum_{i=1}^{N_E} b_i [p]
= \sum_{i=1}^{N_E} \int_{E_i} b_i \ (\tilde{\psi}^{(i)}_l \cdot n) [p] \ ds 
= \int_{\Omega} p \ \nabla \cdot u \ dx,
\end{equation}
where $[p]$ is the jump of $p$ across the coarse edge $E_i$. 
In addition, we have the following estimate for the weighted $L^2$ norm for velocity
\begin{equation}
\label{eq:infsup3}
||u||_{\kappa^{-1}, \Omega}^{2} 
= \int_{\Omega} \kappa^{-1} u \cdot u \ dx
\leq 
\sum_{i=1}^{N_E} 
\int_{\omega_i} \kappa^{-1}  (b_i \tilde{\psi}^{(i)}_r)^2 \ dx
\preceq 
 \left(
\underset{i}{\max} \ \underset{r}{\min} 
\int_{\omega_i} \kappa^{-1} \psi^{(i)}_l \cdot \psi^{(i)}_l \ dx
\right) 
\ H \ \sum_{i=1}^{N_E} b_i^2.
\end{equation}
with the property $\int_{E_i} \psi^{(i)}_l \cdot n \ ds \neq 0$.

Using \eqref{eq:infsup1}, \eqref{eq:infsup3}  and the fact that $\nabla \cdot u = p$, we obtain
\[
||u||_{\kappa^{-1}, \Omega}^{2} + ||\nabla \cdot u||^2_{L^2(\Omega)}
\preceq 
 \left(
\underset{i}{\max} \ \underset{r}{\min} 
\int_{\omega_i} \kappa^{-1} \psi^{(i)}_l \cdot \psi^{(i)}_l \ dx
\right) 
||p||^2_{L^2(\Omega)} + ||p||^2_{L^2(\Omega)},
\]
or
\begin{equation}
\label{eq:infsup4}
||u||_{H(\div; \Omega); \kappa^{-1}}
\preceq 
C_{\text{infsup}}
||p||_{L^2(\Omega)}.
\end{equation}

Using \eqref{eq:infsup2} and \eqref{eq:infsup4}, we have
\[
\int_{\Omega} p \ \nabla \cdot u \ dx 
= \int_{\Omega} p^2 \ dx 
\succeq 
C_{\text{infsup}}^{-1}
||p||_{L^2(\Omega)} 
||u||_{H(\div; \Omega); \kappa^{-1}},
\]
and can obtain the bound \eqref{eq:inf_sup}.
\end{proof}

Finally, we state the prove the convergence of our method. 

\begin{theorem}
Let $u_h$ be the fine-grid solution to \eqref{fine_problem} and $u_H$ be the Mixed GMsFEM solution to \eqref{coarse_problem}. 
Then, the following estimate holds:
\begin{equation}
||u_h - u_H||^2_{\kappa^{-1}, \Omega} 
\preceq 
C_{\text{infsup}}^2 \Lambda^{-1}
\sum_{i=1}^{N_E} a^{(i)}(\hat{u}, \hat{u}) + 
\underset{K\in  \mathcal{T}_H}{\max}(\kappa^{-1}_{\text{min},K})
\sum_{i=1}^{N_C} ||f - \bar{f}||_{L^2(K_i)}^{2}
\end{equation}
where $\Lambda = \underset{i}{\text{min}}\ \lambda_{M^{(i)}+1}^{(i)}$ and $\hat{u}$ is the solution to the projection problem  of \eqref{snap-project}.
\end{theorem}

\begin{proof}
By \eqref{fine_problem} and \eqref{coarse_problem},  and the fact that $V_H\subset V_{h}$, we get that
\begin{equation}
\label{error1}
\begin{split}
\int_{\Omega} \kappa^{-1} (u_h - u_H) \cdot v_H  \ dx
+
\int_{\Omega} (p_h - p_H)  \ \nabla \cdot v_H \ dx
&= 0,  \quad \forall v_H \in V_H,\\
\int_{\Omega} q_H \ \nabla \cdot (u_h - u_H) \ dx 
&=0, \quad \forall q_H \in Q_{H}.
\end{split}
\end{equation}
Combining \eqref{proj_snap},  we have 
\[
\int_K q_H \ \nabla \cdot (u_h - \hat{u}) \ dx 
= \int_K (f - \bar{f}) \ q_H \ dx 
= 0,  \quad \forall q_H \in Q_H.
\]
We then can easily get the following equalities by \eqref{snap-project}, 
\[
\int_{\Omega} p_h \ \nabla \cdot v_H \ dx 
= \int_{\Omega} \hat{p} \ \nabla \cdot v_H \ dx,
\]
noticing that $q_H$ and $\nabla \cdot v_H$ are both constants in any coarse block $K$.

Therefore,  \eqref{error1} can be further written as 
\begin{equation}
\label{error2}
\begin{split}
\int_{\Omega} \kappa^{-1}(u_h - u_H)\cdot v_H  \ dx
- \int_{\Omega} (\hat{p} - p_H) \ \nabla\cdot v_H \ dx 
& = 0,  \quad \forall v_H \in V_H\\
\int_{\Omega} q_H \nabla \cdot (\hat{u} - u_H) \ dx 
& = 0, \quad \forall q_H \in Q_H.
\end{split}
\end{equation}
Since $\hat{u} \in V_{\text{snap}}$, we can rewrite $\hat{u}$ as
\[
\hat{u}
 = \sum_{i=1}^{N_E}\sum_{l=1}^{J^{(i)}} \hat{u}_{ij} \psi_l^{(i)}.
\]
We then define $\hat{u}_{\text{ms}} \in V_H$ as 
\[
\hat{u}_{\text{ms}} 
= \sum_{i=1}^{N_E} \sum_{l=1}^{M^{(i)}} \hat{u}_{ij}  \psi_l^{(i)}.
\]
where $M^{(i)}$ is the number of multiscale basis we select for each coarse neighborhood $\omega_i$. 
	
We further rewrite \eqref{error2} as 
\begin{equation}
\label{error3}
\begin{split}
\int_{\Omega} \kappa^{-1}(u_h - u_H) \cdot v_H \ dx 
- 
\int_{\Omega}(\hat{p}-p_H)  \  \nabla \cdot v_H \ dx 
= 0,  
\quad \forall v_H \in V_H,\\
\int_{\Omega} q_H \ \nabla \cdot (\hat{u}_{\text{ms}} - u_H) \  dx 
= 
\int_{\Omega} q_H \ \nabla \cdot (\hat{u}_{\text{ms}} - \hat{u}) \ dx,  
\quad \forall q_H \in Q_H.
\end{split}
\end{equation}
Let $v_H = \hat{u}_{\text{ms}} - u_H$ and $q_H = \hat{p} - p_H$ and plug back to \eqref{error3}.  
After adding equations up, we get that
\begin{equation}
\label{error4}
\int_{\Omega} \kappa^{-1} (u_h - u_H) \cdot (\hat{u}_{\text{ms}} - u_H) \ dx 
= 
\int_{\Omega} (\hat{p}-p_H) \  \nabla \cdot (\hat{u}_{\text{ms}} - \hat{u}) \ dx. 
\end{equation}

Combining \eqref{error3} and the inf-sup condition \eqref{eq:inf_sup},  we further have To apply the condition, we need $\hat{p} \in V_H$, should we go back to letting $\hat{p}_H  \in V_H $ and $\int \hat{p}_H  = \int \hat{p}$ 
\[
\begin{split}
||\hat{p} - p_H||_{L^2(\Omega)} 
& \leq	
C_{\text{infsup}} \underset{u \in V_H}{\text{sup}} 
\frac{\int_{\Omega} (\hat{p} - p_H) \ \nabla\cdot u \ dx}{||u||_{H(\div; \Omega); \kappa^{-1}}} \\
& = 
C_{\text{infsup}} \underset{u \in V_H}{\text{sup}}
\frac{\int_{\Omega}\kappa^{-1}(u_h - u_H)\cdot u \ dx}{||u||_{H(\div; \Omega); \kappa^{-1}}} \\
& \preceq 
C_{\text{infsup}} \underset{u \in V_H}{\text{sup}}
\frac{||u_h - u_H||_{\kappa^{-1}, \Omega} \ ||u||_{\kappa^{-1}, \Omega}}{||u||_{H(\div; \Omega); \kappa^{-1}}}
\end{split}
\]
Since $V_H$ is a finite dimensional space,  all norms are equivalent. Thus,
\[
||\hat{p} - p_H||_{L^2(\Omega)} 
\preceq  
C_{\text{infsup}} ||u_h-u_H||_{\kappa^{-1}, \Omega}.
\]
	
By the definition of the bilinear form $s^{(i)}(u,v)$
\[
\int_{\Omega} \left( \nabla \cdot (\hat{u}_{\text{ms}} - \hat{u}) \right)^{2} \ dx
\preceq 
\sum_{i=1}^{N_E} 
\int_{\omega_i} \left( \nabla \cdot (\hat{u}_{\text{ms}} - \hat{u}) \right)^{2} \ dx
\preceq 
\sum_{i=1}^{N_E} s^{(i)}(\hat{u}_{\text{ms}} - \hat{u},  \hat{u}_{\text{ms}} - \hat{u})
\]
Then,  by \eqref{error4} and Cauchy–Schwarz inequality,  we can obtain that
\[
\begin{split}
||u_h - u_H||_{\kappa^{-1},\Omega}^2
& = 
\int_{\Omega}\kappa^{-1}(u_h - u_H)^2 \ dx \\
& \leq 
\left| 
\int_{\Omega}\kappa^{-1}(u_h - u_H) (u_h - \hat{u}_{\text{ms}}) \ dx \
\right|  
+ 
\left| 
\int_{\Omega} \kappa^{-1 }(u_h - u_{H}) (\hat{u}_{\text{ms}} - u_H) \ dx \
\right|  \\
&= 
\left| 
\int_{\Omega} \kappa^{-1} (u_h - u_H) (u_h - \hat{u}_{\text{ms}}) \ dx \
\right|
+ 
\left|
\int_{\Omega}(\hat{p} - p_H) \ \nabla \cdot (\hat{u}_{\text{ms}} - \hat{u}) \ dx \ 
\right| \\
& \preceq 
||u_h - u_H||_{\kappa^{-1}, \Omega} \ 
||u_h-\hat{u}_{\text{ms}}||_{\kappa^{-1}, \Omega} 
+ 
|| \nabla\cdot(\hat{u}_{\text{ms}}-\hat{u})||_{L^2(\Omega)} \ 
|| \hat{p} - p_H||_{L^2(\Omega)} \\
& \preceq  
||u_h - u_H||_{\kappa^{-1}, \Omega} \ 
||u_h - \hat{u}_{\text{ms}}||_{\kappa^{-1}, \Omega}
+
\left(
\sum_{i=1}^{N_E} 
s^{(i)}(\hat{u}_{\text{ms}} - \hat{u},  \hat{u}_{\text{ms}} - \hat{u})
\right)^{\frac{1}{2}} \ 
C_{\text{infsup}} ||u_h - u_H||_{\kappa^{-1}, \Omega}
\end{split}
\]

By deviding $||u_h-u_H||_{\kappa^{-1}, \Omega}$ for both sides of the inequality,  we then have 
\[
||u_h - u_H||_{\kappa^{-1}, \Omega}  
\preceq 
||u_h - \hat{u}_{\text{ms}}||_{\kappa^{-1}, \Omega}
+
\left(
\sum_{i=1}^{N_E} 
s^{(i)}(\hat{u}_{\text{ms}} - \hat{u},  \hat{u}_{\text{ms}} - \hat{u})
\right)^{\frac{1}{2}} \ C_{\text{infsup}}
\]
	
For the first term on the right-hand, by triangle inequality, we have 
\[
||u_h - \hat{u}_{\text{ms}}||_{\kappa^{-1}, \Omega}   
\leq 
||u_h - \hat{u}||_{\kappa^{-1}, \Omega}  
+ 
||\hat{u}_{\text{ms}} - \hat{u}||_{\kappa^{-1}, \Omega}  
\]

Moreover, by the definition of the spectral problem \eqref{spectral} we have 
\[
||\hat{u}_{\text{ms}} - \hat{u}||_{\kappa^{-1}, \Omega}  
\preceq 
\sum_{i=1}^{N_E} ||\hat{u}_{\text{ms}} - \hat{u}||_{\kappa^{-1}, \omega_i}  
\preceq 
\sum_{i=1}^{N_E} s^{(i)}(\hat{u}_{\text{ms}} - \hat{u},  \hat{u}_{\text{ms}} - \hat{u}) 
\]

Thus the following inequality holds:
\[
||u_h-u_H||_{\kappa^{-1}, \Omega}   
\preceq 
||u_h-\hat{u}||_{\kappa^{-1}, \Omega}  
+
\left(
\sum_{i=1}^{N_E} s^{(i)}(\hat{u}_{\text{ms}} - \hat{u},  \hat{u}_{\text{ms}} - \hat{u})
\right)^{\frac{1}{2}} \ C_{\text{infsup}}
\]

By \eqref{lemma1} ,we can estimate the first term on the right hand side while the second term can be determined by the fact that eigenfunctions are orthogonal, thus
\[
\begin{split}
s^{(i)}(\hat{u}_{\text{ms}} - \hat{u}, \hat{u}_{\text{ms}} - \hat{u})) 
& = 
s^{(i)}\left( 
\sum_{i=1}^{N_E} \sum_{j = M^{(i)}+1}^{J^{(i)}} \hat{u}_{ij} \psi_j^{(i)},  
\sum_{i=1}^{N_E} \sum_{j = M^{(i)}+1}^{J^{(i)}} \hat{u}_{ij} \psi_j^{(i)} \right)  = 
\sum_{j = M^{(i)}+1}^{J^{(i)}}(\lambda_j^{(i)})^{-1} \ \hat{u}_{ij}^2 \ a^{(i)}(\psi_j^{(i)},  \psi_j^{(i)})\\
& \leq 
(\lambda_{M^{(i)}+1}^{(i)})^{-1} \ a_i(\hat{u}_{\text{ms}} - \hat{u}, \hat{u}_{\text{ms}} - \hat{u})
\leq 
(\lambda_{M^{(i)}+1}^{(i)})^{-1} \ a_i(\hat{u},  \hat{u}),
\end{split}
\]
which gives us the conclusion of this theorem.
\end{proof}

\section{Numerical results}

We consider the numerical simulations of an elliptic equation in three different thin domains $\Omega$:
\begin{itemize}
\item   \textit{Geometry 1}: A rectangular thin domain of size $L_x \times L_y$ with $L_x = 1.0$ and $L_y = 0.1$.
\item  \textit{Geometry 2}: A thin domain with $L_x = 1.0$ and rough top and bottom boundaries with $L_y \in [0.057,0.251]$. 
\item  \textit{Geometry 3}: Y - shaped thin domain with varying thickness. 
\end{itemize}

\begin{figure}[h!]
\centering
\includegraphics[width=0.49\linewidth]{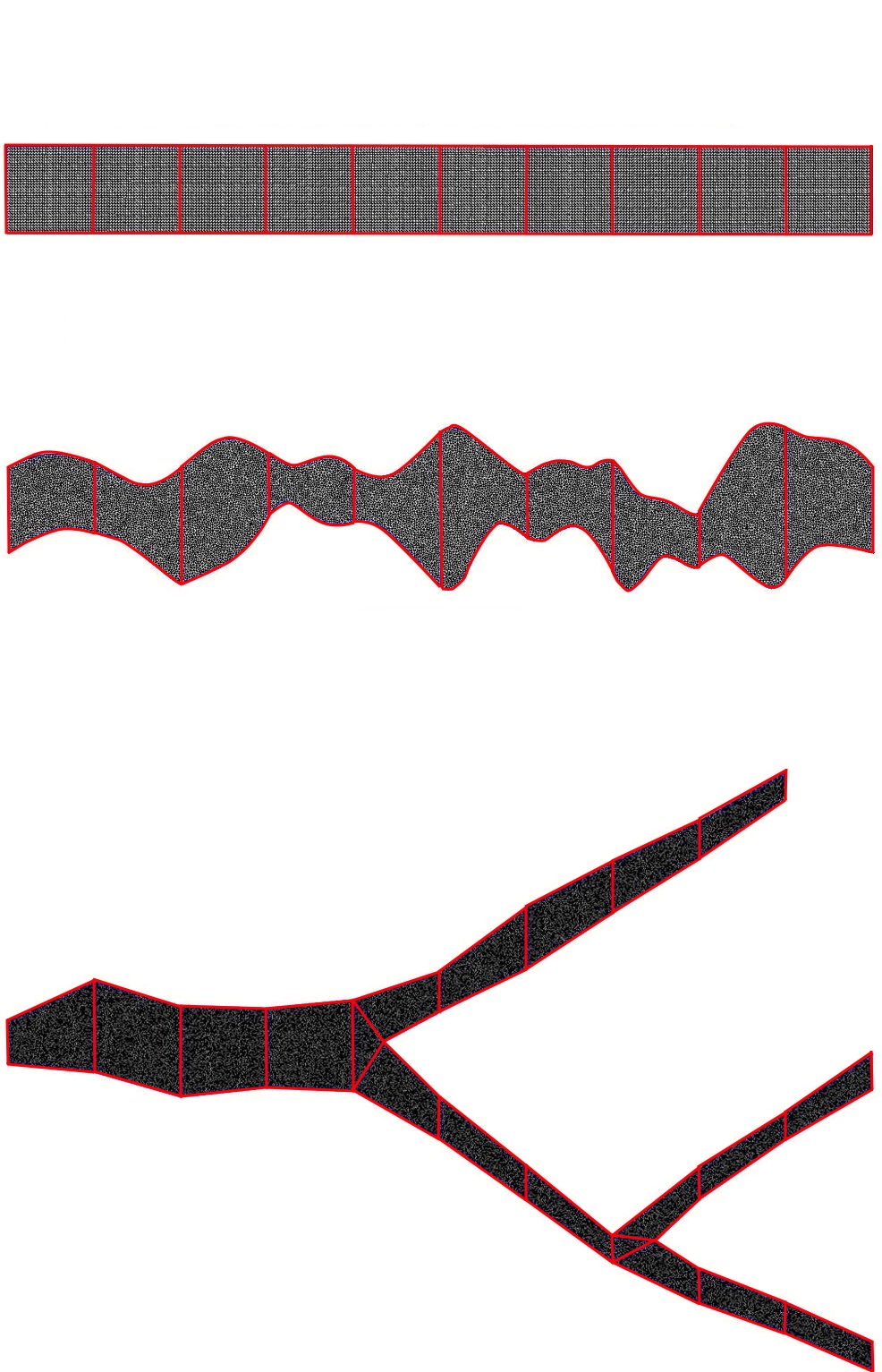} 
\includegraphics[width=0.49\linewidth]{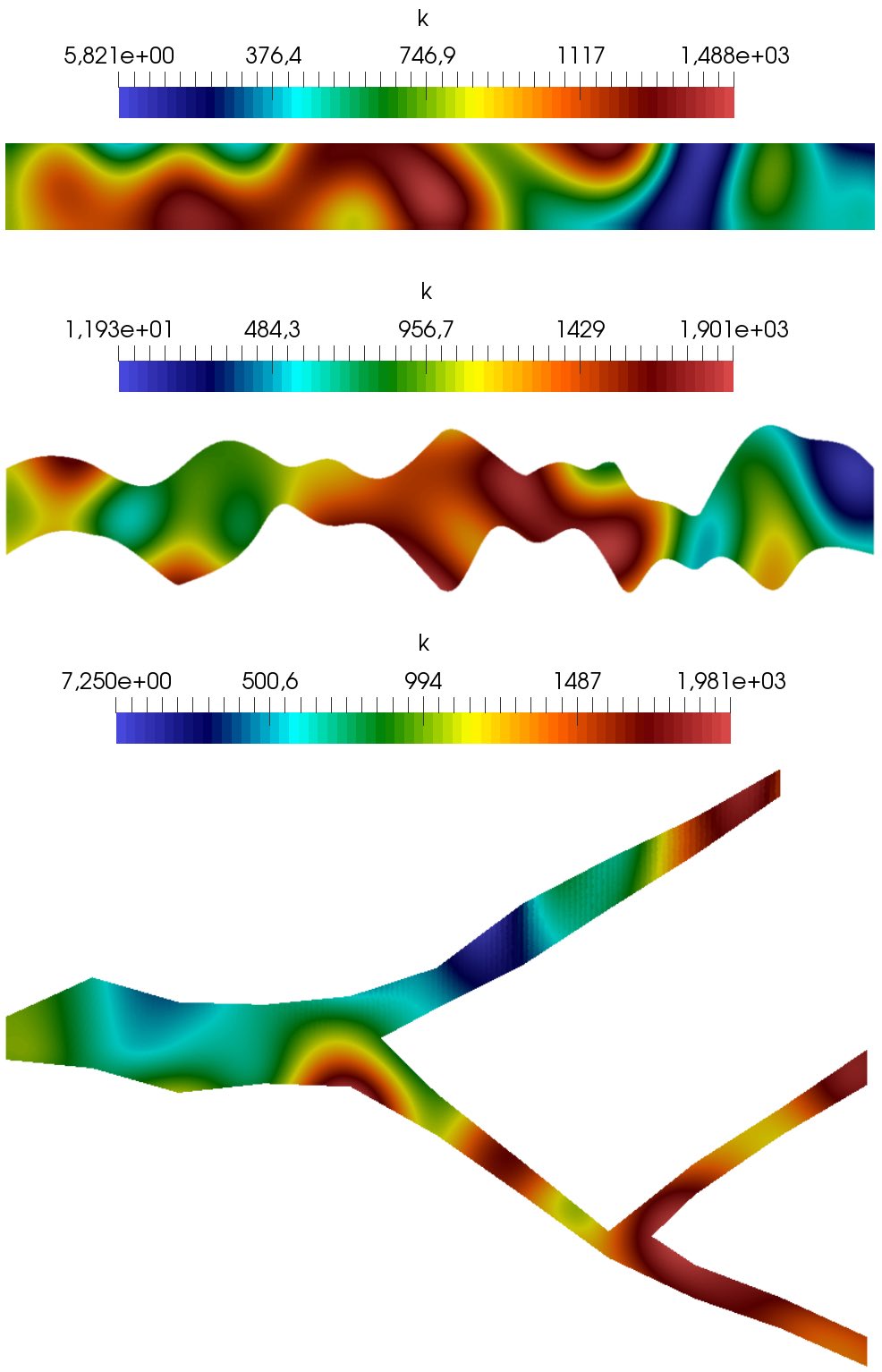} 
\caption{Computational domains with fine grid and  heterogeneous coefficient $k$ for \textit{Geometry 1}(first row), \textit{Geometry 2}(second row) and \textit{Geometry 3}(third row). 
First column:  fine grid.  
Second column: heterogeneous coefficient. }
\label{pic6}
\end{figure}

For \textit{Geometry 1} we use a structured triangular $160 \times 16$ fine grid with $31072$ facets and $20480$ cells. 
For \textit{Geometry 2,} and \textit{Geometry 3}, we consider unstructured fine grids with following parameters: 
$42951$ facets, $28351$ cells (\textit{Geometry 2});  
$95217$ facets, $62819$ cells (\textit{Geometry 3}). 
The computational domains, fine and coarse grids are presented in Figure \ref{pic6}. For \textit{Geometry 1}, we use a structured rectangular coarse grid with 10  coarse cells. For \textit{Geometry 2}, we use a coarse grid with 10 cells and 11 facets and for \textit{Geometry 3}, we use a coarse grid with $20$ cells and $23$ facets.   Heterogeneous coefficient $k$ is depicted in Figure \ref{pic6} (second column). 

We consider two test cases with different boundary conditions and source terms:
\begin{itemize}
\item \textit{Test 1}: We set $p_1=1$ on the inlet boundary $\Gamma _{in}$,  $p_2 = 0$ on the outlet boundary $\Gamma _{out}$ and $f=0$.
\item \textit{Test 2}: We set $p_1=0$ on the inlet and outlet boundaries $\Gamma _{in} \cup \Gamma _{out}$ and $f=1$.
\end{itemize}
On the top and bottom boundaries $\Gamma _w$ we set zero flux boundary condition for both test cases.

We calculate the relative $L^2$ error for pressure on the coarse grid and error for velocity on the fine grid as follows
\[
e_p^H = \frac{||p_H - \bar{p}_h||_{L^2(\Omega)}}{||\bar{p}_h||_{L^2(\Omega)}} \times 100 \%,  \quad
e_u^h = \frac{||u_H - u_h||_{\kappa^{-1}, \Omega}}{||u_h||_{\kappa^{-1}, \Omega}} \times 100 \%, 
\]
where 
$(u_h, p_h)$ are the reference fine-grid solutions,  
$(u_H, p_H)$ are the multiscale solutions using Mixed GMsFEM,  
$\bar{p}_h$ is the reference pressure on coarse grid (coarse cell average for $p_h$).

\begin{figure}[h!]
\centering
\includegraphics[width=0.9\linewidth]{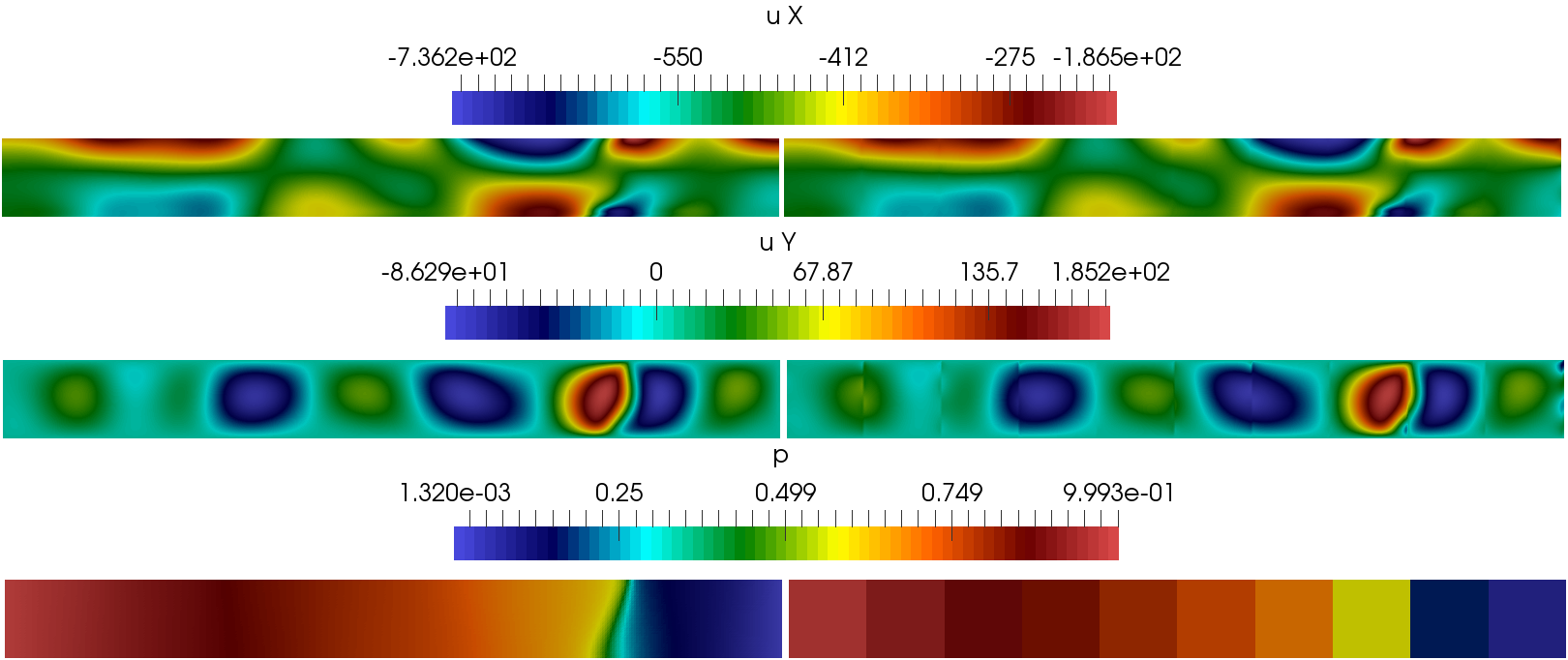} 
\caption{Fine grid solution (left column) and multiscale solution using $4$ multiscale basis functions (right column). \textit{Geometry 1, Test 1}}
\label{res1}
\end{figure}

\begin{figure}[h!]
\centering
\includegraphics[width=0.9\linewidth]{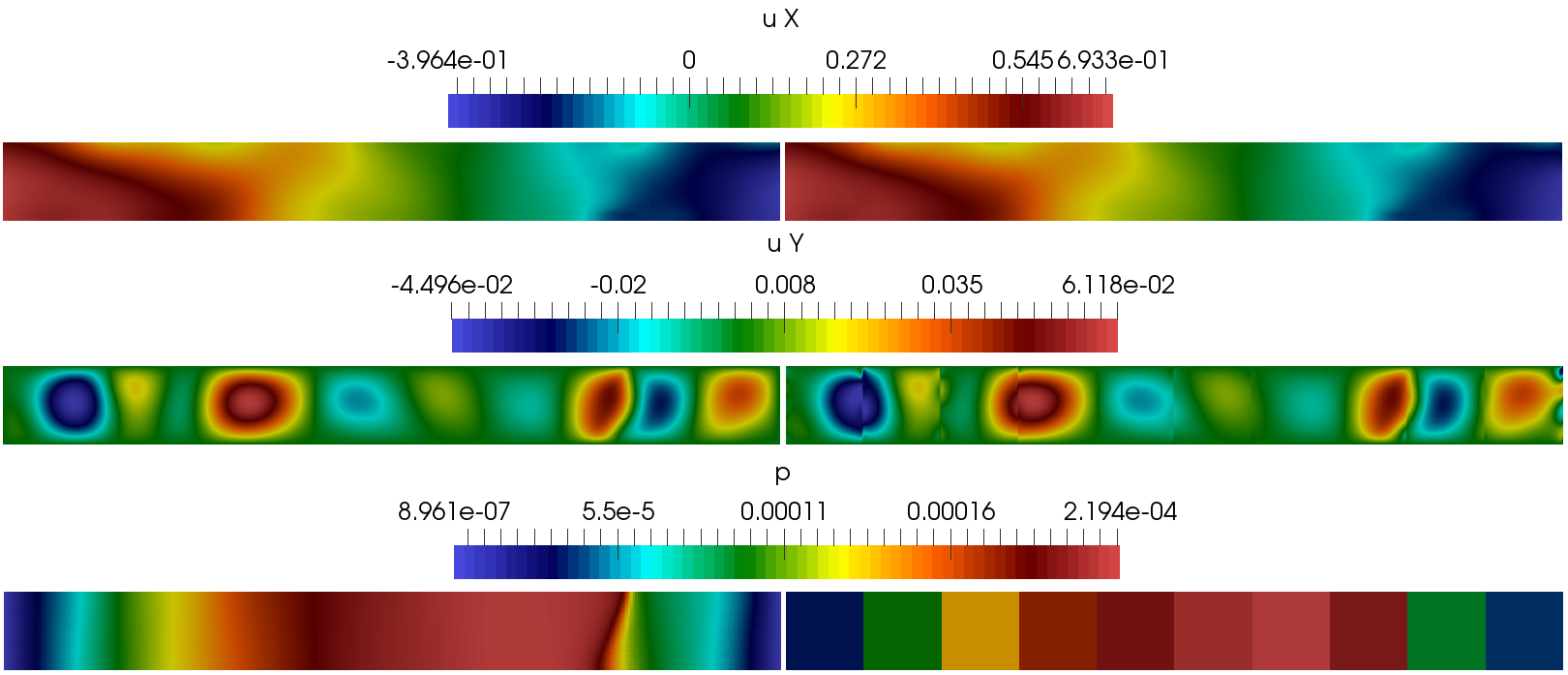} 
\caption{Fine grid solution (left column) and multiscale solution using 4  multiscale basis functions (right column). \textit{Geometry 1, Test 2}}
\label{res1.1}
\end{figure}

\begin{table}[h!]
\begin{center}
\begin{tabular}{|c|c|cc|}
\hline
$M$ & $DOF_c$ & $e_u^h$, \% & $e_p^H$, \%  \\
\hline
1 & 21 & 14.824 & 0.883 \\
2 & 32 & 3.066 & 0.114 \\
4 & 54 & 0.613 & 0.008 \\
8 & 98 & 0.161 & 0.001 \\
12 & 142 & 0.105 & 0.001 \\
\hline
\end{tabular}
\,\,\,\,
\begin{tabular}{|c|c|cc|}
\hline
$M$ & $DOF_c$ & $e_u^h$, \% & $e_p^H$, \%  \\
\hline
1 & 21 & 12.034 & 2.094 \\
2 & 32 & 3.249 & 0.232 \\
4 & 54 & 0.614 & 0.013 \\
8 & 98 & 0.146 & 0.001 \\
12 & 142 & 0.071 & 0.001 \\
\hline
\end{tabular}
\end{center}
\caption{Numerical results for  \textit{Geometry 1}. Relative error in $L^2$ norm for different number of multiscale basis functions. Left: \textit{Test 1}. Right: \textit{Test 2}.}
\label{t1}
\end{table}

We start multiscale modeling for test problems in \textit{Geometry 1}. 
Fine scale and multiscale solutions using 4 multiscale basis function are presented in Figure \ref{res1} for \textit{Test 1} and in Figure \ref{res1.1} for \textit{Test 2}.  In Table \ref{t1} we present the relative errors  for \textit{Test 1} and \textit{Test 2}. In this table we show the errors associated with different number of multiscale basis functions. Here, $DOF_c=M \times N_E$ and $DOF_f$ denotes the size of the multiscale and fine grid solutions ($DOF_f = 51 552$), $M$ is the number of  multiscale basis functions. From this tables, we can see that Mixed GMsFEM can provide solutions with high accuracy in both cases. We can observe that the accuracy of the method is almost the same for both sets of input parameters: \textit{Test 1} and \textit{Test 2}.  In fact, using $4$ multiscale basis functions is sufficient to obtain a good solution. We can provide the error less than $1 \%$, using only $0.1 \%$ degrees of freedom from solution on the fine grid. We can also get even better accuracy of the Mixed GMsFEM by increasing the number of multiscale basis functions in each local domain $\omega _i$.

\begin{figure}[h!]
\centering
\includegraphics[width=0.9\linewidth]{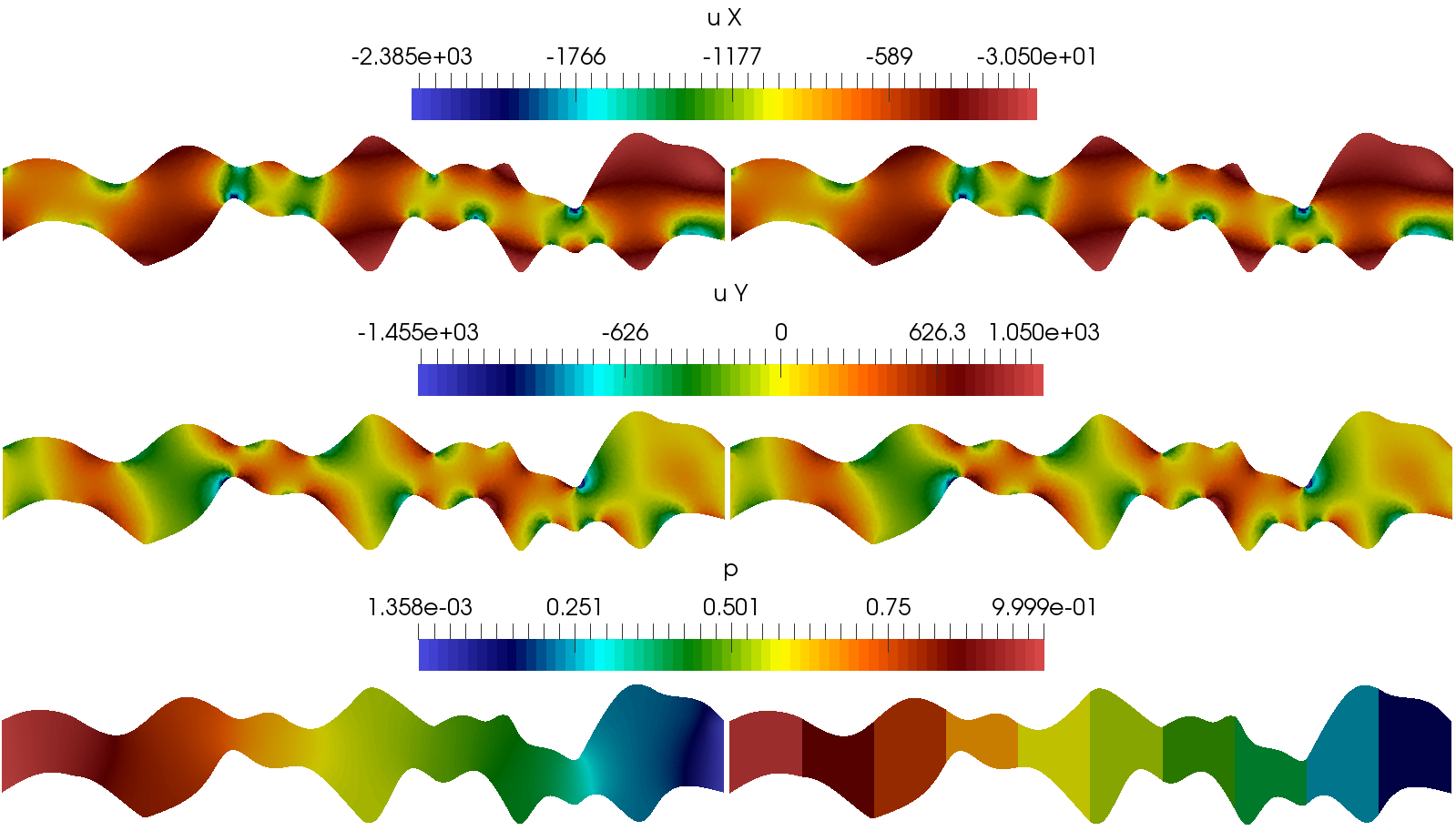} 
\caption{Fine grid solution (right column) and multiscale solution using 8  multiscale basis functions (left column). \textit{Geometry 2, Test 1.}}
\label{res2}
\end{figure}

\begin{figure}[h!]
\centering
\includegraphics[width=0.9\linewidth]{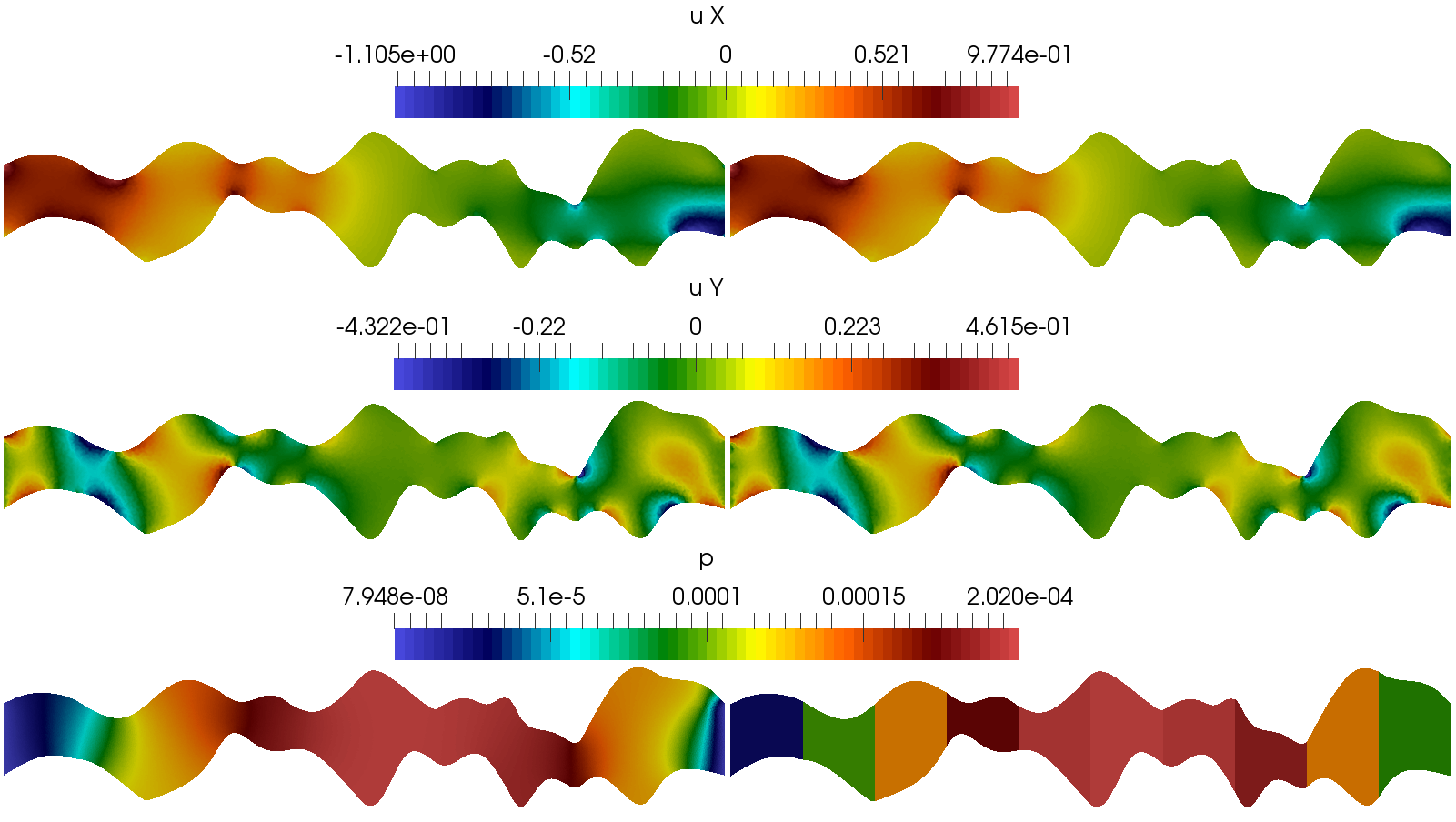} 
\caption{Fine grid solution (right column) and multiscale solution using 8  multiscale basis functions (left column). \textit{Geometry 2, Test 2.}}
\label{res2.1}
\end{figure}

\begin{table}[h!]
\begin{center}
\begin{tabular}{|c|c|cc|}
\hline
$M$ & $DOF_c$ & $e_u^h$, \% & $e_p^H$, \%   \\
\hline
1 & 21 & 22.032 & 6.358 \\
2 & 32 & 14.021 & 2.468 \\
4 & 54 & 3.201 & 0.165 \\
8 & 98 & 0.863 & 0.007 \\
12 & 142 & 0.664 & 0.006 \\
\hline
\end{tabular}
\,\,\,\,
\begin{tabular}{|c|c|cc|}
\hline
$M$ & $DOF_c$ & $e_u^h$, \% & $e_p^H$, \%   \\
\hline
1 & 21 & 25.235 & 12.077 \\
2 & 32 & 14.018 & 4.023 \\
4 & 54 & 4.299 & 0.299 \\
8 & 98 & 1.335 & 0.025 \\
12 & 142 & 1.063 & 0.015 \\
\hline
\end{tabular}
\end{center}
\caption{Numerical results for  \textit{Geometry 2}. Relative error in $L^2$ norm for different number of multiscale basis functions. Left: \textit{Test 1}. Right: \textit{Test 2}.}
\label{t2}
\end{table}

The multiscale modeling of flow in \textit{Geometry 2} is a very interesting case, because we can test Mixed GMsFEM associated to local domains with complex shapes. We completed the multiscale modeling for this case, and the numerical results are presented in Figure \ref{res2} and in \ref{res2.1} for \textit{Test 1} and \textit{Test 2} respectively. For both test problems, we use $8$ multiscale basis functions in each local domain.   In Table \ref{t2} we demonstrate the relative errors of the multiscale solution compared to the fine grid solution ($DOF_f = 71302$) when using different number of multiscale basis functions where in the left table are the errors for\textit{Test 1} and in the right table are errors for \textit{Test 2}. The behavior of the errors is similar to the results with the previous geometry, but in this case, we need to use $8$ multiscale basis functions to obtain a good solution. Similar to past results,  we need to use only $0.1 \%$ degrees of freedom from solution on the fine grid to obtain the error less than $1 \%$.  The increase in the error in this case is justified by the fact that the computational domain has a more complex shape. We can also observe an improvement in the accuracy of the method with an increase in the number of multiscale basis functions.  As we can see, Mixed GMsFEM perfectly solves tasks in domains with such complex shape, which shows that the proposed multiscale basis functions can describe the properties of the media with different shape of local domains very well.   


\begin{table}[h!]
\begin{center}
\begin{tabular}{|c|c|cc|}
\hline
$M$ & $DOF_c$ & $e_u^h$, \% & $e_p^H$, \%   \\
\hline
1 & 42 & 16.175 & 1.202 \\
2 & 64 & 5.345 & 0.162 \\
4 & 108 & 2.214 &  0.021 \\
8 & 196 & 1.005 & 0.017 \\
12 & 284 & 0.698 & 0.017 \\
\hline
\end{tabular}
\,\,\,\,
\begin{tabular}{|c|c|cc|}
\hline
$M$ & $DOF_c$ & $e_u^h$, \% & $e_p^H$, \%   \\
\hline
1 & 42 & 14.263 & 2.131 \\
2 & 64 & 4.125 & 0.168 \\
4 & 108 & 2.063 & 0.045 \\
8 & 196 & 0.969 & 0.013 \\
12 & 284 & 0.614 & 0.011 \\
\hline
\end{tabular}
\end{center}
\caption{Numerical results for  \textit{Geometry 3}. Relative error in $L^2$ norm for different number of multiscale basis functions. Left: \textit{Test 1}.  Right: \textit{Test 2}.}
\label{t3}
\end{table}

Finally, we consider the numerical experiment in the more complicated Y-tube domain (\textit{Geometry 3}).  Numerical result for \textit{Test 1} and \textit{Test 2} using $8$ multiscale basis function are presented in Figures \ref{res3} and \ref{res3.1} respectively. In Table \ref{t3}, we present the relative errors between multiscale solution and fine grid solution ($DOF_f = 158036$) for different number of multiscale basis functions for both test problems. We can observe good accuracy of the Mixed GMsFEM solutions in this case using $8$ multiscale basis functions.   In this case,  we also need to use only $0.1 \%$ degrees of freedom from solution on the fine grid to obtain the error around $1 \%$. We also observe the convergence in error with an increase in the number of multiscale bases. The behavior of the solution is similar to the previous results. So we can see, that the Mixed GMsFEM provide a good solution for problems even in the computational field with such a complex shape. The results suggest that this method can be applied to simulate fluid flow in applications such as fluid flow in blood vessels.    

\begin{figure}[h!]
\centering
\includegraphics[width=0.9\linewidth]{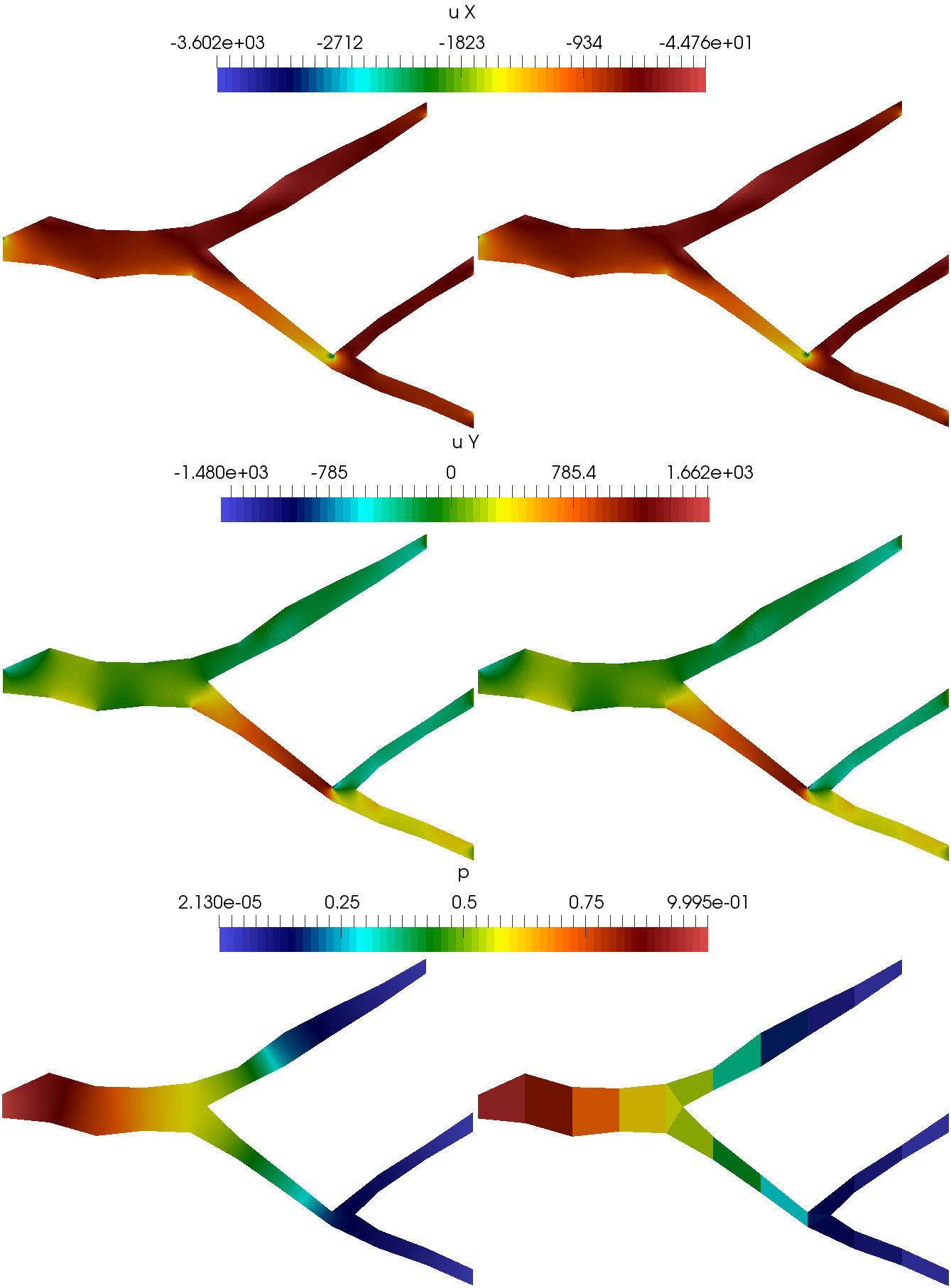} 
\caption{Fine grid solution (left column) and multiscale solution using $8$ multiscale basis functions (right column). \textit{Geometry 3}, \textit{Test 1}.}
\label{res3}
\end{figure}
\FloatBarrier

\begin{figure}[h!]
\centering
\includegraphics[width=0.9\linewidth]{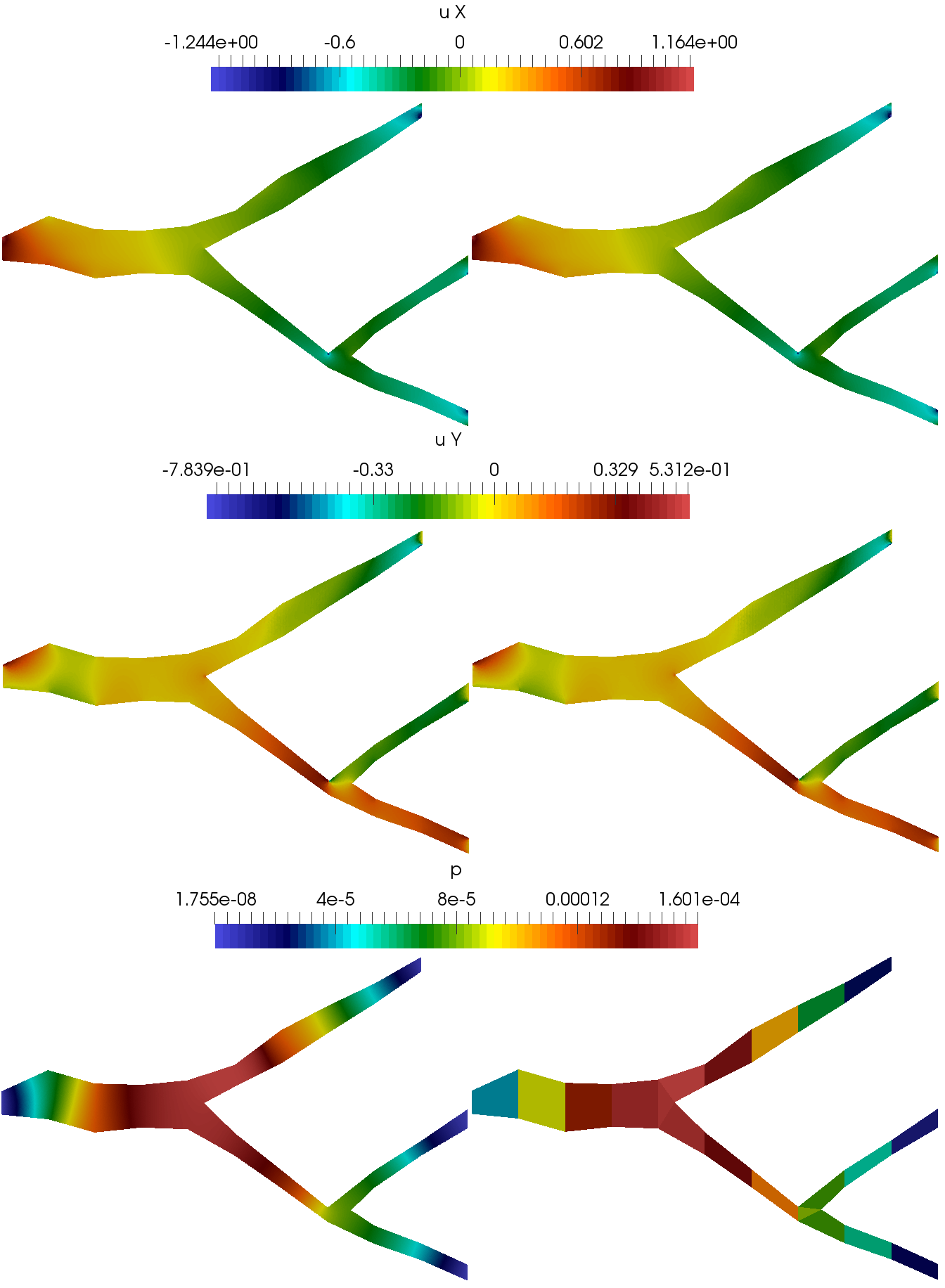} 
\caption{Fine grid solution (left column) and multiscale solution using $8$ multiscale basis functions (right column). \textit{Geometry 3}, \textit{Test 2}.}
\label{res3.1}
\end{figure}
\FloatBarrier


\section{Conclusions}

We developed an algorithm based on the Mixed Generalized Finite Element Method for flow problem in thin domains. Modeling was carried out in thin domains with different shapes, including domains with complex geometries. We also conducted a study over the accuracy of the multiscale solution associated with different number of multiscale basis functions.  The obtained results show good applicability of the method to this type of problems. The proposed multiscale model reduction technique allows us to describe the heterogeneities of the domain on micro-level with high accuracy and further lead to multiscale solutions with small errors while reducing the size of the system significantly.

\section{Acknowledgements}
This work is supported by the mega-grant of the Russian Federation Government N14.Y26.31.0013 and RFBR N19-31-90066.
The research of Eric Chung is partially supported by the Hong Kong RGC General Research Fund (Project numbers 14304719 and 14302620) and CUHK Faculty of Science Direct Grant 2020-21.

\bibliographystyle{unsrt}
\bibliography{lit}

\end{document}